\documentclass[a4paper,12pt,reqno]{amsart}
 \usepackage[english]{babel}
  \usepackage{amsmath,amsfonts,amssymb,amsthm,bbm}
   \usepackage{graphics,epsfig,psfrag} 
   \usepackage{paralist,subfigure}

   \usepackage{hyperref} 

    \usepackage[active]{srcltx} 
    \usepackage{color,soul} 
    \usepackage[latin1]{inputenc}
    \usepackage[OT1]{fontenc}

\newtheorem{theorem}{Theorem}[section]

\newtheorem{proposition}[theorem]{Proposition}

\theoremstyle{definition}
\newtheorem{definition}[theorem]{Definition}
\newtheorem{remark}{Remark}

\DeclareMathOperator{\dist}{dist}

\DeclareMathOperator\supp{supp}

\def\R{\mathbb{R}}

\let\vp=\varphi

\let\t=\tilde
\let\ol=\overline
\let\ul=\underline
\let\.=\cdot
\let\0=\emptyset

\let\mc=\mathcal

\def\1{\mathbbm{1}}

\def\thm#1{Theorem~\ref{thm:#1}}

\def\limt{\lim_{t\to+\infty}}

\newenvironment{formula}[1]{\begin{equation}\label{#1}}
                       {\end{equation}\noindent}

\def\Fi#1{\begin{formula}{#1}}
\def\Ff{\end{formula}\noindent}

\newcommand{\SE}{\setcounter{equation}{0} \section}

\def\re{\mc{R}^e}
\def\ri{\mc{R}^i}
\def\ar{asymptotically spherical}

\title[\tiny Symmetrization and anti-symmetrization in parabolic equations]{\bf Symmetrization and anti-symmetrization \\ in parabolic equations}
\author{Luca  Rossi}
\address{CNRS, Ecole des Hautes Etudes en Sciences Sociales, PSL Research University,  
Centre d'Analyse et Math\'ematiques Sociales, 190-198 avenue de France 
F-75244 Paris Cedex 13, France}
\email{luca.rossi@ehess.fr}
\thanks{}

\begin{document}

\begin{abstract}
	We derive some symmetrization and anti-symmetrization properties of parabolic equations.
	First, we deduce from a result by Jones \cite{Jones}  a quantitative estimate of how far  the
	level sets of solutions are from being spherical. Next, using this property, we derive a
	criterion providing solutions whose level sets do not converge to spheres for a class of equations including linear
	equations and Fisher-KPP reaction-diffusion equations.
\end{abstract}

\maketitle

\vspace{-17pt}

\smallskip
\SE{Introduction}

We are concerned with the spherical symmetrization feature of the
equation
\Fi{evol}
\partial_t u=\Delta u+f(u),\quad t>0,\ x\in\R^N.
\Ff
This is a semilinear reaction-diffusion equation, but we do not exclude the case where~$f$ is linear.
A result by Jones \cite{Jones} asserts that solutions emerging from compactly supported initial data 
look more and more spherical as t increases, in the following sense: the normal to 
the level sets at nonsingular points always intersect the convex hull of the support of the datum and therefore, if the level sets go to infinity, 
their normal approaches the radial direction.
This result is derived for reaction-diffusion equations of bistable type, but the very elegant 
proof, based on a reflection argument, actually applies to much more general equations, also 
time-dependent. Using some geometrical arguments, we will show that Jones' result implies more than 
the convergence of the normal to the radial direction: it provides an explicit estimate of the distance between the upper level sets and suitable balls. Namely, if $\supp u_0\subset B_\delta$ then the upper 
level set 
\Fi{uls}
\mc{U}_\theta(t):=\{x\in\R^N \,:\, u(t,x)>\theta\}
\Ff
satisfies 
\Fi{Hd}
B_{r(t)}\subset\mc{U}_\theta(t)\subset B_{r(t)+\delta\pi},
\Ff
for some function $r$ and for $t$ sufficiently large. The precise statement is given in \thm{sym} below.
This means that the Hausdorff distance between the upper level set~$\mc{U}_\theta(t)$ and a 
suitable ball is bounded by the constant $\delta\pi/2$ for large $t$. In the case where $f$ is of KPP-type
and time-independent, this property has already been
obtained by Ducrot \cite{Ducrot}, with in addition the explicit 
expression for $r(t)$, but with a generic constant instead of the precise 
value $\delta\pi/2$.

Then, the question that naturally arises is whether property \eqref{Hd} is sharp
or not, that is, does the difference between the radii of the internal and 
external balls tends to $0$ as $t\to+\infty$\,? Of course, for this question to 
make sense we have to consider all possible balls, not just the ones centred at 
the 
origin. This leads us to define
\begin{equation}\label{Ri}
\ri_\theta(t):=\sup\{r>0\,:\, \exists x_0\in\R^N,\ u(t,x)>\theta\text{ for 
all } x\in B_r(x_0)\},
\end{equation}
\begin{equation}\label{Re}
\re_\theta(t):=\inf\{r>0\,:\, \exists x_0\in\R^N,\ u(t,x)\leq\theta\text{ for all
} x\in (B_r(x_0))^c\},
\end{equation}
which are the radii respectively of the largest ball contained in 
$\mc{U}_\theta(t)$ and of the smallest ball containing 
$\mc{U}_\theta(t)$. Another notion of symmetrization is the 
convergence to a radial function. We consider both.

\begin{definition}\label{def:sym}
 We say that a function $u:[0,+\infty)\times\R^N\to\R$ is {\em asymptotically
spherical} if it fulfils one of the following properties:
\begin{enumerate}[(i)]
\item 
$$\forall\theta\in(0,1),\quad 
\limt\big(\re_\theta(t)-\ri_\theta(t)\big)=0;$$
 \item there exist two functions $\phi:[0,+\infty)^2\to\R$,
$\Gamma:[0,+\infty)\to\R^N$ such that
$$\lim_{t\to+\infty}\big(u(t,x)-\phi(t,|x-\Gamma(t)|)\big)=0.$$
\end{enumerate}
\end{definition}
The symmetrization properties (i) and (ii) are not related in general, unless 
the function $r\mapsto\phi(t,r)$ has a strictly monotonic character. 

Two counter-examples to the spherical symmetrization are known in the 
literature for reaction-diffusion equations of bistable type: Yagisita~\cite{vsJones0} and Roussier~\cite{vsJones2}, the latter in dimension~2. The common idea there is to 
construct a solution which looks like a planar front when followed along a 
given  direction, shifted by different values depending on the direction. This is possible due to the strong stability of the (unique up to shift) front. Hence, 
those examples have a rather specific form. 
The question remains open in many relevant cases, such as, strikingly, the linear one.

In the present paper we focus on concave (in a weak sense) terms $f$, including the case of linear equations and Fisher-KPP equations, possibly time-dependent. One cannot proceed as in the 
bistable case because of the lack of strong stability of fronts. Using a 
different method, we build a large class of non-\ar\ 
solutions. We find in particular that the set of initial data for 
which the solution is not asymptotically 
spherical is dense in the space of compactly supported continuous functions.
Curiously, to achieve this we use our previous symmetrization result.
We also need to derive some estimates of the distance between level sets of
solutions - c.f.~Theorems \ref{steep}, \ref{steep2} below - that we believe
are of independent interest. 


\subsection{Hypotheses and main results}

Consider the problem
\Fi{evol-t}
  \partial_t u=\Delta u+f(t,u),\quad t>0,\ x\in\R^N.\\
\Ff
We will always assume in the sequel that 
$f(t,z)$ is H\"older continuous in $t$ and uniformly
Lipschitz continuous in $z$, uniformly with respect to $t$, and satisfies
$$\forall t>0,\quad f(t,0)=0.$$
The initial datum will always be nonnegative and continuous and solutions 
will be classical and locally bounded in time. 
We are concerned with {\em invading} solutions, that is, 
solutions satisfying
\Fi{invasion}
\exists Z\in(0,+\infty],\quad
\forall K\Subset\R^N,\quad \liminf_{t\to+\infty}\Big(\min_{x\in K}u(t,x)\Big)\geq Z.
\Ff
If $Z$ is finite, it can be viewed as the saturation level for the problem. Typically, $Z=1$ in reaction-diffusion equations; if $f$ satisfies the KPP hypothesis~\eqref{KPP} below, it is known that any nontrivial solution is invading, 
but this may not be the case for other classes of reaction terms, c.f.~\cite{AW}. 
In the case where $f$ is linear and increasing in $z$, we have that all nontrivial solutions fulfil the
invasion condition with $Z=+\infty$.  

Our symmetrization result holds as soon as Jones' technique applies and thus it does not require any specific assumption on $f$. 

\begin{theorem}\label{thm:sym}
Let $u$ be a solution to \eqref{evol-t} with initial 
datum supported in a ball~$B_\delta$ satisfying the invasion property \eqref{invasion}. Then,
for $\theta\in(0,Z)$ and $t$ large enough, the upper level set $\mc{U}_\theta(t)$ defined by \eqref{uls} is star-shaped with respect to the 
origin and satisfies
\Fi{spherical}
B_{r_\theta(t)}\subset \mc{U}_\theta(t) \subset 
B_{r_\theta(t)+\delta\pi},
\Ff
for some positive function $r_\theta$. In particular, 
$$
0\leq\re_\theta(t)-\ri_\theta(t)\leq\delta\pi.
$$
\end{theorem}

Next, we build non-asymptotically spherical solutions under the following positivity and weak concavity assumption
on $f$:
\Fi{KPP}
\begin{cases}
	\displaystyle\exists Z>0,\quad \forall z\in(0,Z),\quad
	\inf_{t>0}f(t,z)>0,\\ 
	\displaystyle z\mapsto\frac{f(t,z)}{z}\ \hbox{ is 
		nonincreasing in $(0,+\infty)$, for all $t>0$}.
\end{cases}
\Ff
This hypothesis holds in the linear case $f(t,z)=\zeta(t)z$ with $\inf \zeta>0$, or when $f$ is a  
reaction term of KPP-type, such as $z(1-z)$.

\begin{theorem}\label{gen-anti}
Assume that $f$ satisfies \eqref{KPP}. Let $u_1,u_2$ 
be two nonnegative, not identically equal to $0$, continuous functions with
compact support. Then, for $|\xi|$ large enough, the solution to \eqref{evol-t} 
with initial datum 
$$u_0(x)=u_1(x)+u_2(x+\xi)$$
is not asymptotically spherical.
\end{theorem}

This theorem roughly says that if the support of the initial datum 
has two  components which are far apart enough then the 
solution is not \ar. Thus, adding to a compactly supported initial 
datum $u_1$ a compactly supported perturbation $u_2$, as small as wanted but 
sufficiently far, will give rise to a solution which is not asymptotically 
spherical. 


\SE{The symmetrization result}

\begin{proof}[Proof of \thm{sym}]
 We know from Jones \cite{Jones} that, for $t$ sufficiently large,
if $u(t, x_0) = \theta$ and $\nabla u(t, x_0)\neq 0$ then the normal line
to the upper level set $\mc{U}_\theta(t)$ through the point $x_0$ intersects
the 
convex hull of the support of the initial datum~$u_0$. This property is derived in \cite{Jones}
using a reflection argument (see also a simplified proof by Berestycki \cite[Theorem 
2.9]{NATO}) inspired by Serrin \cite{Serrin} and Gidas, Ni, Nirenberg \cite{GNN}. 
We repeat the first step of the argument,  in 
order to see that $u$ is radially decreasing outside $B_\delta\supset \supp 
u_0$.
More precisely, we will show that
\Fi{decreasing}
\forall x_0\notin B_\delta,\quad x_0\.\nabla u(t, x_0)<0.
\Ff
Take $x_0\notin 
B_\delta$ and, calling $\mc{T}$ the reflection with 
respect to the hyperplane $\{x\in\R^N\,:\, x\.x_0=|x_0|^2\}$, define
$$v(t,x):=u(t,\mc{T}(x)).$$
The function $v$ satisfies the same equation as $u$, because the Laplace operator is invariant under reflection, together with 
the initial and boundary conditions
$$v(0,x)=0\quad\text{for }x\.x_0<|x_0|^2,\qquad
v(t,x)=u(t,x)\quad\text{for }t>0,\ x\.x_0=|x_0|^2.$$
It follows from the parabolic strong comparison principle that, for $t>0$ and 
$x\.x_0<|x_0|^2$, it holds $v(t,x)<u(t,x)$. Then, by Hopf's lemma,
$$\partial_{x_0} u(t,x_0)<\partial_{x_0} v(t,x_0).$$
From this, because $\partial_{x_0} v(t,x_0)=-\partial_{x_0} u(t,x_0)$, we 
get $\partial_{x_0} u(t,x_0)<0$, that is \eqref{decreasing}.

Fix $\theta\in(0,Z)$. The invasion property \eqref{invasion} implies that $\mc{U}_\theta(t)\supset B_\delta$ for $t$ larger than some $t_0$. Therefore, for $t>t_0$ and $e\in S^{N-1}$, the 
function $\rho\mapsto u(t,\rho e)$ is larger than $\theta$ for $\rho\in(0,\delta)$. Then, by 
\eqref{decreasing}, it is strictly decreasing for $\rho\geq\delta$ and moreover it tends to 
$0$ at infinity because $u(t,x)\to0$ as $|x|\to\infty$ by standard 
parabolic decay. Consequently, the set $\mc{U}_\theta(t)$ is star-shaped with respect to the 
origin for $t>t_0$.

Now, take $t>t_0$ large enough so that Jones' result applies.
Pick two points $P,Q\in\partial\mc{U}_\theta(t)$. By considering the 
plane $H$ through $P$, $Q$ and the origin, we can reduce to the bidimensional case: 
ignoring the other $N-2$ directions, we write 
$$H\cap \partial\mc{U}_\theta(t)=\{\vp(\alpha)(\cos\alpha,\sin\alpha)\,:\, 
\alpha\in[0,2\pi)\},$$
for some positive function $\vp$. The points
$P$, $Q$ are obtained for two angles $\alpha_P$, $\alpha_Q$. It is 
not restrictive to assume that $|\alpha_P-\alpha_Q|\leq\pi$. Since 
$B_\delta\subset\mc{U}_\theta(t)$, we know from~\eqref{decreasing} that $\vp$ is of class $C^1$. For $\alpha\in[0,2\pi)$, we set 
for short 
$\mathbf{x}=\vp(\alpha)(\cos\alpha,\sin\alpha)$, 
$\mathbf{v}=\nabla u(\mathbf{x})$ and we compute
\Fi{ortogonal}
	0=\frac{d}{d\alpha} u(\vp(\alpha)(\cos\alpha,\sin\alpha))=
\mathbf{v}\.(\vp'(\alpha)\cos\alpha-\vp(\alpha)\sin\alpha,
\vp'(\alpha)\sin\alpha+\vp(\alpha)\cos\alpha).
\Ff
Jones' result implies that the distance from the line $s\mapsto 
\mathbf{x}-s\mathbf{v}$ and the origin is less than $\t\delta:=\max\{|x|\,:\, 
x\in\supp u_0\}<\delta$.
Namely,
$$\mathbf{w}:=\mathbf{x}-\frac{\mathbf{x}\.\mathbf{v}}
{|\mathbf{v}|^2}\,\mathbf{v}$$ satisfies
$|\mathbf{w}|\leq\t\delta$.
In order to get an estimate on $\vp'$ we observe that, by \eqref{ortogonal},
\[\begin{split}
0 &=(\mathbf{x}-\mathbf{w})\.(\vp'(\alpha)\cos\alpha-\vp(\alpha)\sin\alpha,
\vp'(\alpha)\sin\alpha+\vp(\alpha)\cos\alpha)\\
&=\vp(\alpha)\vp'(\alpha)-\mathbf{w}
\.(\vp'(\alpha)\cos\alpha-\vp(\alpha)\sin\alpha,
\vp'(\alpha)\sin\alpha+\vp(\alpha)\cos\alpha),
\end{split}
\]
whence, because $|\mathbf{w}|\leq\t\delta$,
$$\big(\vp(\alpha)\vp'(\alpha)\big)^2\leq
\t\delta^2[(\vp'(\alpha))^2+(\vp(\alpha))^2].$$
We eventually find that
$$(\vp'(\alpha))^2\leq \frac{\t\delta^2}{1-\t\delta^2/(\vp(\alpha))^2}.$$
Notice that, by the invasion condition, $\min\vp\to+\infty$ as $t\to+\infty$.
Hence, for $t$ large enough, 
$|\vp'(\alpha)|\leq \delta$ for all $\alpha\in[0,2\pi)$.
Reverting to the points $P$, $Q$, this implies that 
$\big||P|-|Q|\big|\leq\pi\delta$. This concludes the proof of the theorem.
\end{proof}


\smallskip
\SE{Non-\ar\ solutions}

In order to make the construction of the counter-example of Theorem \ref{gen-anti}
as transparent as possible, we start with the particular instance where $u_1\equiv u_2$.

We will need to control the distance between level sets of solutions. Because of its independent interest, we derive it under weaker assumptions than \eqref{KPP}. Namely,
\Fi{f>g}
\begin{cases}
	\displaystyle
\exists g\in C([0,+\infty)),\quad \forall t,z\geq0,\ f(t,z)\geq g(z),\\
\displaystyle \exists Z>\theta_0>0,\quad g\leq0\ \text{ in }(0,\theta_0), \quad 
g>0\ \text{ in }(\theta_0,Z),\quad \int_0^Z g>0.
\end{cases}
\Ff
%
Besides the case \eqref{KPP}, this hypothesis is fulfilled when $f$ is a reaction term of any of the classical types considered in the literature: monostable, combustion, bistable \cite{AW}, but~also in much wider
cases where $f$ has several zeroes. It is only required here to ensure invasion for solutions with large enough, compactly supported initial data. 

\begin{theorem}\label{steep}
Under the assumption \eqref{f>g}, let $u$ 
be a solution of \eqref{evol-t} with compactly supported initial datum for 
which the invasion property \eqref{invasion} holds with the same $Z$ as in \eqref{f>g}. Then, for $\theta_0<\theta'<\theta<Z$, the functions $r_{\theta}$, $r_{\theta'}$ provided by \thm{sym} satisfy
$$\liminf_{t\to+\infty}\big(r_{\theta'}(t)-r_{\theta}(t)\big)<+\infty.$$
\end{theorem}


Let us comment on this statement before giving the proof. Combined with~\eqref{spherical}, it implies that
the Hausdorff distance between the upper level sets of $u$ satisfies 
$$\liminf_{t\to+\infty}\,\dist\big(\mc{U}_{\theta'}(t),\mc{U}_\theta(t)\big)<+\infty.$$
This means that the width of the interface $\{\theta'\!<\!u\!<\!\theta\}$ is bounded along a sequence of times, which can be viewed as a steepness property of the profile.
We do not know whether the estimate holds true with $\liminf$ replaced by $\sup$.
This is left as an open question.
A positive answer is given in \cite{Zlatos-width} when $f$ is of combustion-type and is allowed
to have some $x$-dependence, but not $t$-dependence, in dimension $N\leq3$, together with a counter-example in dimension $N>3$.
Let us mention that the interface $\{\theta'\!<\!u\!<\!\theta\}$ may not have
bounded width 
if the initial datum does not decay sufficiently fast at infinity, see
\cite[Theorem 8.4]{Uchi}. Also, the restriction on the levels $\theta$, $\theta'$ to belong to  
the same positivity region of $f$ cannot be dropped, since we know from 
\cite[Theorem 3.3]{FMcL} that, for multistable nonlinearities, different level
sets can spread with different speeds\,\footnote{ This is proved for front-like initial
data, but then the same phenomenon is expected to occur for invading
solutions with compactly supported initial data.}.

\begin{proof}[Proof of Theorem \ref{steep}]
We first show that the variation of $r_{\theta}(t)$ with respect to $t$ cannot tend to infinity, next we see that the variation with respect to $\theta$ is controlled by that with respect to 
$t$.

\smallskip
{\em Step 1.} Control of the variation with respect to $t$.\\
We claim that
\Fi{variation}
\forall \theta\in(\theta_0,1),\ T>0,\quad
\liminf_{t\to+\infty}\big(r_{\theta}(t+T)-r_{\theta}(t)\big)<+\infty.
\Ff
The function $u$ is a subsolution of the linear equation
$$\partial_t u=\Delta u+\zeta u,\quad t>0,\ x\in\R^N,$$
with time-independent zero-order term $\zeta:=\sup_{t,z>0}f(t,z)/z$.
It is straightforward to check (using for instance the heat kernel) that such equation admits a finite speed of spreading~$c^*$, which can  be actually computed:
$c^*=2\sqrt{\zeta}$. Namely, being $u$ a subsolution with compactly supported initial datum, there holds
$$\forall c>c^*,\quad
\lim_{t\to+\infty}\sup_{|x|\geq ct}u(t,x)=0.
$$
Since $u(t,x)>\theta$ if $|x|<r_{\theta}(t)$, this implies that, for any $c>c^*$,  $r_{\theta}(t)<ct$ for $t$ sufficiently large.
Suppose now that \eqref{variation} does not hold. Then there exist $ \theta\in(\theta_0,1)$ and $T,\tau>0$ such that 
$$\forall t\geq\tau,\quad
r_{\theta}(t+T)-r_{\theta}(t)>(c^*+1)T,
$$
which, applied recursively yields
$$\liminf_{n\to\infty}\frac{r_{\theta}(\tau+nT)}{nT}
\geq c^*+1.$$ 
This contradicts the fact that $r_{\theta}(t)<(c^*+1/2)t$ for large $t$.
%
%
%
%

\smallskip
{\em Step 2.} Control of the variation with respect to $\theta$.\\
Let $\theta_0$, $Z$, $g$ be from \eqref{f>g} and take $\theta_0<\theta'<\theta<Z$. It is clear that, up to 
perturbing the function $g$ in $(0,\theta_0)$ and then in a neighbourhood of $0$ and $Z$, it is not restrictive to assume that
$g<0$ in $(0,\theta_0)$ and $g(0)=g(Z)=0$, still preserving property \eqref{f>g}.
Then, it is known that the invasion occurs for solutions to
the equation
\Fi{subeq}
\partial_s v=\Delta v+g(v),\quad s>0,\ x\in\R^N,
\Ff
with large initial data.
Namely, by \cite[Remark 6.5]{AW} there exists $R>0$ such that 
the solution $v$ with initial datum
$v_0(x)=\theta'\1_{B_R}(x)$ converges locally uniformly to $1$ as $s\to+\infty$.
In particular, $v(T,0)>\theta$ for some $T>0$. 
Since $u$ invades, there exists $\tau>0$ such that $r_{\theta'}(t)>R$
for $t\geq\tau$. Take $t\geq\tau$, $\rho\in(R,r_{\theta'}(t))$ and $\xi\in 
B_{\rho-R}$. The inequality $u(t,\xi+x)\geq v_0(x)$ holds for all $x\in\R^N$, 
because if $|x|<R$ then $|\xi+x|<\rho<r_{\theta'}(t)$ and 
hence $u(t,\xi+x)>\theta'\geq v_0(x)$, whereas
$v_0(x)=0\leq u(t,\xi+x)$ if $|x|\geq R$.
Therefore, since any space/time 
translation of $u$ is a supersolution to~\eqref{subeq}, the parabolic comparison 
principle yields 
$$\forall s\geq0,\ x\in\R^N,\quad
u(t+s,\xi+x)\geq v(s,x).$$
Computed at $s=T$, $x=0$, this inequality gives
$u(t+T,\xi)\geq v(T,0)>\theta$. This means that 
$\xi\in\mc{U}_\theta(t+T)$, which by \eqref{spherical} is contained in 
$B_{r_\theta(t+T)+\delta\pi}$, where $B_\delta$ contains the support of the initial datum of $u$. Thus, by the arbitrariness of $\xi\in B_{\rho-R}$ and 
$\rho\in(R,r_{\theta'}(t))$, we deduce 
\Fi{space-variation}
\forall t\geq\tau,\quad
r_{\theta'}(t)-R\leq r_\theta(t+T)+\delta\pi.
\Ff
Owing to \eqref{variation}, this concludes the proof of the lemma.
\end{proof}

\begin{remark}\label{rk:steep}
In the case where $f$ is independent of $t$, or more in general satisfies
$f(\.+T,\.)\geq f$ for some $T>0$,
Theorem \ref{steep} can be improved to obtain a relation between the level sets
of two distinct solutions, relaxing at the same time the restriction on the initial datum. Namely, 
if $u^1$, $u^2$ are invading solutions with initial data $u^1_0$, $u^2_0$ smaller than $Z$ and decaying at most as $e^{-|x|^2}$,
then for all $\theta,\theta'\in(\theta_0,Z)$, the functions 
$r_\theta^1$, $r_{\theta'}^2$ associated with $u^1$,
$u^2$  satisfy
\Fi{steep2tindep}
\liminf_{t\to+\infty}\big|r_\theta^1(t)-r_{\theta'}^2(t)\big|
<+\infty.
\Ff
%
%
%
Indeed, by \eqref{invasion}, at time $nT$ with $n$ large enough, both $u^1$ and $u^2$ are larger than the function~$v_0$ used in the 
step 2 of the proof of Theorem \ref{steep}.
Thus, if $f$ fulfils $f(\.+T,\.)\geq f$ then the solution $u$ emerging from $v_0$ satisfies
$$\forall t\geq0,\ x\in\R^N,\quad
u(t,x)\leq \min\{u^1(t+nT,x),u^2(t+nT,x)\}.$$
On the other hand, $u$ is invading and, by comparison with the heat equation, decays not faster than $e^{-\frac{|x|^2}t}$. As a consequence, at time $mT$ with $m$ large, it is greater than $u^1_0$, $u^2_0$ and therefore
$$\forall t\geq0,\ x\in\R^N,\quad
u(t+mT,x)\geq \max\{u^1(t,x),u^2(t,x)\}.$$
These estimates imply that the level sets of $u^1$ and $u^2$ are trapped between  those of~$u$, up to a time shift.
One can then derive \eqref{steep2tindep} applying Theorem \ref{steep} - more precisely properties \eqref{space-variation} and \eqref{variation} - to $u$.
\end{remark}


\begin{proposition}\label{pro:twins}
The conclusion of Theorem \ref{gen-anti} holds if $u_1\equiv u_2$.
\end{proposition}

\begin{proof}
First, under the standing assumptions, the existence of a unique classical bounded solution for the Cauchy problem associated with \eqref{evol-t} follows from the standard parabolic theory. Furthermore, 
by comparison with the equation with reaction term $\inf_{t>0}f(t,z)$, we infer 
that under the KPP assumption \eqref{KPP} any nontrivial solution to \eqref{evol-t}
satisfies the invasion property \eqref{invasion} with $Z$ given by \eqref{KPP}, see \cite{AW}.

Let $w$ be the solution to \eqref{evol-t} with initial datum $u_1$, and let
$\t\delta$ be such that $\supp u_1\subset B_{\t\delta}$. 
Call $(\t r_\theta)_{\theta\in(0,Z)}$ the family of functions $r_\theta$ provided by
\thm{sym} for the solution~$w$. 
We know from the  parabolic strong comparison principle that $u(t,x)> 
w(t,x)$ for all
$t>0$, $x\in\R^N$. Thus, for given $\theta\in(0,Z)$ and $t$  
large enough so that $\t r_\theta(t)$ is 
defined, we have that
$$\min_{|x|\leq\t r_\theta(t)} u(t,x)>\min_{|x|\leq\t r_\theta(t)} w(t,x)\geq\theta,$$
whence, in particular, $u(t,\t r_\theta(t)\xi/|\xi|)>\theta$. Comparing now 
$u$ with $w(t,x+\xi)$ we get
$$\min_{|x+\xi|\leq\t r_\theta(t)}u(t,x)>\min_{|x+\xi|\leq\t r_\theta(t)}w(t,x+\xi)\geq\theta,$$
and then $u(t,-\xi-\t r_\theta(t)\xi/|\xi|)>\theta$. It follows that the diameter 
of the upper level set~$\mc{U}_\theta(t)$ is at least $|\xi|+2\t r_\theta(t)$, 
which means that 
\Fi{re>}
\re_\theta(t)\geq \t r_\theta(t)+\frac{|\xi|}2.
\Ff
%
%
%
%

Next, we derive an upper bound for $\ri_{\theta}(t)$ by
considering the
expansion of the level set in a direction $\eta\in
S^{N-1}$ orthogonal to $\xi$. It is just here that we use the concavity
hypothesis \eqref{KPP}: it implies that the sum of supersolutions is a
supersolution, because
\Fi{s+s}
\forall\, 0<\alpha\leq\beta,\quad
f(t,\alpha+\beta)\leq
\frac{f(t,\beta)}\beta(\alpha\!+\!\beta)\!=\!\frac{f(t,\beta)}
\beta\alpha\!+\!f(t,\beta)\!\leq\!f(t,\alpha)\!
+\!f(t,\beta).
\Ff
Hence, the function $w(t,x)+w(t,x+\xi)$
is a supersolution of \eqref{evol-t}, coinciding with $u$ at 
$t=0$, and then, by
comparison, $u(t,x)\leq w(t,x)+w(t,x+\xi)$ for $t\geq0$, $x\in\R^N$.
We deduce that if $x\.\eta\geq \t r_{\theta/2}(t)+\t\delta\pi$ then
$|x|,|x+\xi|\geq \t r_{\theta/2}(t)+\t\delta\pi$ and thus
$$u(t,x)\leq w(t,x)+w(t,x+\xi)\leq\theta.$$
This means that the width of the set $\mc{U}_{\theta}(t)$ is at most 
$2\t r_{\theta/2}(t)+2\t\delta\pi$, whence
\Fi{ri<}
\ri_{\theta}(t)\leq \t r_{\theta/2}(t)+\t\delta\pi.
\Ff

Gathering together \eqref{re>} and \eqref{ri<} we eventually obtain
$$\liminf_{t\to+\infty}
\big(\ri_{\theta}(t)-\re_{\theta}(t)\big)\leq\liminf_{t\to+\infty}
\big(\t r_{\theta/2}(t)-\t r_{\theta}(t)\big)+\t\delta\pi-\frac{|\xi|}2.$$
Now, we know from Theorem \ref{steep} that the first term of the right-hand
side 
is finite (and independent of $\xi$) and therefore $u$ 
does not fulfil condition (i) of Definition~\ref{def:sym} provided $|\xi|$ is 
sufficiently large. 

Finally, if we use \eqref{ri<} with $\theta$ replaced by any 
$\theta'\in(0,\theta)$, we get 
$$\liminf_{t\to+\infty}\big(\ri_{\theta'}(t)-\re_{\theta}(t)\big)<0,$$
for a possibly larger $|\xi|$. This contradicts the property (ii) of Definition~\ref{def:sym}, because the latter implies that $\ri_{\theta'}(t)> \re_\theta(t)$ for $t$ large enough.
To see this, consider $\phi$ and $\Gamma$ given by  Definition~\ref{def:sym} (ii). Take $\t\theta\in(\theta',\theta)$ and define
$$\rho(t):=\max\{r>0\,:\,\phi(t,r)=\t\theta\}.$$
This quantity is well defined for $t$ large enough because $u(t,x)-\phi(t,|x-\Gamma(t)|)\to0$ as $t\to+\infty$ uniformly in $x$ and $u$ is invading. There holds that
$$\phi(t,r)\begin{cases}
=\t\theta & \text{for }r=\rho(t)\\
<\t\theta & \text{for }r>\rho(t),
\end{cases}$$
and therefore, for $t$ sufficiently large,
$$\partial B_{\rho(t)}(\Gamma(t))\subset\mc{U}_{\theta'}(t),\qquad
\mc{U}_{\theta}(t)\subset \ol B_{\rho(t)}(\Gamma(t)).$$
The second inclusion implies that $\re_\theta(t)\leq\rho(t)$.
On the other hand, for $t$ large enough, we have that $\mc{U}_{\theta'}(t)$ is star-shaped owing to \thm{sym} and thus the first inclusion yields $\ol B_{\rho(t)}(\Gamma(t))\subset
\mc{U}_{\theta'}(t)$, whence $\ri_{\theta'}(t)\geq\rho(t)> \re_\theta(t)$.
\end{proof}

\subsection{The general construction}
This section is devoted to the proof of Theorem~\ref{gen-anti}. First, we need
an improvement of Theorem \ref{steep} which allows one to compare the position
of level sets of distinct solutions. We have seen in Remark \ref{rk:steep}
that this can be achieved with minor modification in the time-independent case. 
For the time-dependent equation \eqref{evol-t}, an alternative argument is required. We are able to perform it under the KPP hypothesis.

\begin{theorem}\label{steep2}
Assume that $f$ satisfies \eqref{KPP}.
Let $u^1$, $u^2$ be two solutions of \eqref{evol-t} with 
compactly
supported initial data 
$u^1_0, u^2_0\geq0,\not\equiv0$.
Then, for all $\theta,\theta'\in(0,1)$, the functions 
$r_\theta^1$, $r_{\theta'}^2$ given by \thm{sym} with $u=u_1$ and 
$u=u_2$ respectively, satisfy
$$\liminf_{t\to+\infty}\big|r_\theta^1(t)-r_{\theta'}^2(t)\big|
<+\infty.$$
\end{theorem}

\begin{proof}
We let $\tau_y$ denote the translation acting on a (possibly time independent) 
function $u$ as $\tau_y u(t,x):=u(t,x-y)$. Take $\zeta\in\R^N$ such that 
$\ul u_0:=\min\{\tau_{\zeta} u^1_0,u^2_0\}$ is not identically equal to $0$,
and let $\ul u$ be the solution of \eqref{evol-t} with initial datum $\ul u_0$.
Then, let $\ol u$ be the solution of 
\eqref{evol-t} with initial datum $\ol u_0:=\max\{u^1_0,u^2_0\}$. By comparison, 
we have that
\Fi{minmax}
\ul u\leq \tau_{\zeta}u^1,u^2,\qquad \ol u\geq u^1,u^2.
\Ff
Set $\ul\theta:=\min\{\theta,\theta'\}$, $\ol\theta:=\max\{\theta,\theta'\}$ and 
let $\ul r_{\ol\theta},\ol r_{\ul\theta}$ denote the functions provided by 
\thm{sym} with $u=\ul u$ and $u=\ol u$ respectively (recall that invasion always occurs by \eqref{KPP}). Our goal is to 
bound $r_\theta^1, r_{\theta'}^2$ from below by $\ul r_{\ol\theta}$ (up to additive constants) and from above by
$\ol r_{\ul\theta}$ and finally to control the difference between
$\ol r_{\ul\theta}$ and $\ul r_{\ol\theta}$. Of course, the last 
step will be the most involved, and this is where we require the KPP hypothesis
\eqref{KPP}, together with Theorem \ref{steep}.

Let $t$ be large enough so that \eqref{spherical} applies for the various functions and values of~$\theta$ in play.
From the first inequality in \eqref{minmax} and \eqref{spherical} we infer from one 
hand that 
$$|x|= r_{\theta'}^2(t)+\delta_2\pi\implies 
\ul u(t,x)\leq u^2(t,x)\leq \theta'\leq\ol\theta
\implies |x|\geq \ul r_{\ol\theta}(t),$$
that is, 
\Fi{est1a}
\ul r_{\ol\theta}(t)\leq r_{\theta'}^2(t)+\delta_2\pi.
\Ff
From the other hand,
$$|x|=|\zeta|+r_{\theta}^1(t)+\delta_1\pi\implies 
\ul u(t,x)\leq \tau_{\zeta}u^1(t,x)\leq \theta\leq\ol\theta
\implies |x|\geq \ul r_{\ol\theta}(t),$$
whence
\Fi{est1b}
\ul r_{\ol\theta}(t)\leq |\zeta|+r_{\theta}^1(t)+\delta_1\pi.
\Ff
Similarly, the second inequality in \eqref{minmax} and \eqref{spherical} yield
\Fi{est2}
r_{\theta}^1(t),r_{\theta'}^2(t)\leq \ol r_{\ul\theta}(t)+\ol\delta\pi,
\Ff
where $\ol\delta:=\max\{\delta_1,\delta_2\}$ (observe that $\supp \ol u_0\subset 
B_{\ol\delta}$).

We now estimate $\ol r_{\ol\theta}$ in terms of $\ul r_{\ul\theta}$. Since $\ul 
u_0\not\equiv0$ and $\ol u_0$ is compactly supported, one can find a 
family of points $\{x_1,\dots,x_n\}$ such that 
$$\ol u_0\leq \sum_{j=1}^n\tau_{x_j}\ul u_0.$$
Recall that the KPP hypothesis yields \eqref{s+s}, which, applied recursively, 
implies that the sum of supersolutions is a supersolution. Therefore, by 
comparison, 
$$\ol u\leq \sum_{j=1}^n\tau_{x_j}\ul u$$
holds true for all $t>0$.
Then, with the same argument as before (notice that $\supp\ul u_0\subset 
B_{\delta_2}$) we find that
$$|x|=\max_{j=1,\dots, n}|x_j|+\ul r_{\ul\theta/n}(t)+\delta_2\pi\implies 
\ol u(t,x)\leq \sum_{j=1}^n\tau_{x_j}\ul u(t,x)\leq 
\sum_{j=1}^n\ul\theta/n=\ul\theta,$$
whence 
$$ \ol r_{\ul\theta}(t)\leq \max_{j=1,\dots, n}|x_j|+\ul 
r_{\ul\theta/n}(t)+\delta_2\pi.$$
Therefore, owing to Theorem \ref{steep},
$$\liminf_{t\to+\infty}\big(\ol r_{\ul\theta}(t)-\ul r_{\ol\theta}(t)\big)\leq
\max_{j=1,\dots, n}|x_j|+\delta_2\pi+\liminf_{t\to+\infty}\big(
\ul r_{\ul\theta/n}(t)-\ul r_{\ol\theta}(t)\big)<+\infty.$$

The proof is thereby concluded, because, by \eqref{est1a}-\eqref{est2},
\[\big|r_\theta^1-r_{\theta'}^2\big|=\max 
\{r_\theta^1,r_{\theta'}^2\}-\min\{r_\theta^1,r_{\theta'}^2\}\leq
\ol r_{\ul\theta}+\ol\delta\pi-\ul r_{\ol\theta}+\delta_2\pi+|\zeta|+\delta_1\pi.\qquad
\qedhere
\]
\end{proof}

\begin{proof}[Proof of Theorem \ref{gen-anti}]
We want to adapt the arguments of the proof of Proposition~\ref{pro:twins}.
First, we consider a nonnegative, not identically 
equal to $0$, continuous functions $\ul w_0,$ with compact support, 
satisfying, for some $x_1,x_2\in\R^N$, 
$$\forall x\in\R^N,\quad 
\ul w_0(x)\leq u_1(x+x_1)\,,\, u_2(x+x_2).$$
Let $\ul w$ be the solution emerging from $\ul w_0$ and $\ol w$ be the one emerging from  $\max\{u_1,u_2\}$, and call 
$(\ul r_\theta)_{\theta\in(0,1)}$, $(\ol r_\theta)_{\theta\in(0,1)}$ the families of functions provided by
\thm{sym} associated with these solutions. The comparison principle yields, for $t>0$ and $x\in\R^N$, 
$$u(t,x)\leq \ol w(t,x)+\ol w(t,x+\xi),\qquad
u(t,x)\geq\max\big\{ \ul w(t,x-x_1)\,,\, \ul w(t,x-x_2+\xi)\big\}.$$
For $\theta\in(0,1)$, using the first inequality above together with the argument employed to derive \eqref{ri<} we find that $\ri_{\theta}(t)\leq \ol r_{\theta/2}(t)+\ol\delta\pi$,
while from the second inequality, computed at $x=x_1+\ul r_\theta(t)\xi/|\xi|$ and  
$x=x_2-\xi-\ul r_\theta(t)\xi/|\xi|$, we deduce 
$$\re_\theta(t)\geq\ul r_\theta(t)+\frac{|\xi|-|x_1|-|x_2|}2.$$
Consequently,
$$\liminf_{t\to+\infty}
\big(\ri_{\theta}(t)-\re_{\theta}(t)\big)\leq\liminf_{t\to+\infty}
\big(\ol r_{\theta/2}(t)-\ul 
r_{\theta}(t)\big)+\ol\delta\pi-\frac{|\xi|-|x_1|-|x_2|}2.$$
We now make use of Theorem \ref{steep2}. It entails that the first term of the 
right-hand side is finite and therefore $u$ 
does not fulfil condition (i) of Definition~\ref{def:sym} if $|\xi|$ is 
sufficiently large. 
Clearly, one can get the above inequality with $\ri_{\theta}(t)$ and $\ol 
r_{\theta/2}$ replaced by $\ri_{\theta'}(t)$ and $\ol r_{\theta'/2}$ 
respectively, for any $\theta'\in(0,\theta)$. Then, choosing $|\xi|$ very large one infers that 
 condition (ii) of Definition~\ref{def:sym} is violated as well because, as seen 
at the end of the proof of Proposition \ref{pro:twins}, it entails $\ri_{\theta'}(t)> \re_\theta(t)$ for $t$ large enough.
\end{proof}

\end{document}